\documentclass[11pt,a4paper,english,reqno]{amsart}
\usepackage[a4paper,footskip=1.5em,margin=2.5cm]{geometry}
\usepackage{amsmath,amssymb,amsthm,mathtools}
\usepackage[mathscr]{euscript}
\usepackage[usenames,dvipsnames]{color}
\usepackage{adjustbox,tikz,calc,graphics,babel,standalone}
\usepackage{bbm}
\usepackage{subcaption}
\usepackage{csquotes,enumerate,verbatim}
\usepackage[final]{microtype}
\usepackage[numbers]{natbib}
\usetikzlibrary{shapes.misc,calc,intersections,patterns,decorations.pathreplacing}
\usepackage{hyperref}
\hypersetup{colorlinks=true,linkcolor=blue,citecolor=blue,pdfpagemode=UseNone,pdfstartview={XYZ null null 1.00}}
\usepackage{cmtiup}

\theoremstyle{plain}
\newtheorem*{theorem*}{Theorem}
\newtheorem{theorem}{Theorem}[section]

\newtheorem{proposition}[theorem]{Proposition}
\newtheorem*{claim*}{Claim}

\newtheorem{conjecture}[theorem]{Conjecture}
\newtheorem{problem}[theorem]{Problem}
\newtheorem{question}[theorem]{Question}
\theoremstyle{remark}

\newcommand{\F}{\mathcal{F}}
\newcommand{\G}{\mathcal{G}}

\let\emptyset\varnothing

\let\originalleft\left
\let\originalright\right
\renewcommand{\left}{\mathopen{}\mathclose\bgroup\originalleft}
\renewcommand{\right}{\aftergroup\egroup\originalright}

\begin{document}

\title{Refuting conjectures in extremal combinatorics via linear programming}
\author{Adam Zsolt Wagner}
\thanks{Department of Mathematics, ETH, Z\"urich, Switzerland. Email:
\href{mailto:zsolt.wagner@math.ethz.ch} {\nolinkurl{zsolt.wagner@math.ethz.ch}}.}

\maketitle
\begin{abstract}
    We apply simple linear programming methods and an LP solver to refute a number of open conjectures in extremal combinatorics.

\end{abstract}

\section{Introduction}

Conjectures in extremal combinatorics are often intricate -- it can be easy to miss a better construction than the one we have, or to misjudge whether the conjecture is true or false for other reasons. Any general method that can tell us whether a statement is likely to be true or false can be extremely useful in practice.

In the present manuscript we argue that the use of linear programming and LP solvers is such a method in extremal combinatorics. We hope to convince the reader of its usefulness by using it to resolve a number of open conjectures, questions and problems from a variety of areas.
\begin{itemize}
    \item In Section~\ref{sec:linbas} we give the basics of linear programming and several examples on how to phrase questions in extremal combinatorics as linear programs.
    \item In Section~\ref{sec:results} we give our main results. 
    \begin{itemize}
        \item In~\ref{subsec:anti} we disprove a claim and a conjecture of Frankl on the size of antichains of fixed diameter.
        \item In~\ref{subsec:int} we disprove several conjectures of Frankl and Huang on the diversity of intersecting set systems.
        \item In~\ref{subsec:kat} we disprove two conjectures of Katona, and one conjecture of Frankl et al, on  multipartite generalizations of the Erd\H{o}s--Ko--Rado theorem.
        \item In~\ref{subsec:anstee} we solve a problem of Anstee related to a forbidden configuration in set systems.
        \item In~\ref{subsec:hoffman} we answer a question of Ihringer--Kupavskii on regular set systems achieving a Hoffman-type bound.
         \item In~\ref{subsec:kleitman} we disprove a conjecture of Frankl--Tokushige related to the Kleitman matching problem.
        \item In~\ref{subsec:aha} we disprove a conjecture of Aharoni--Howard on bipartite graphs without rainbow matchings.
        \item In~\ref{subsec:jessica} we improve a construction given by De Silva et al related to a Tur\'an-type problem.
    \end{itemize}
\end{itemize}

\section{Basics and examples of linear programming in combinatorics}\label{sec:linbas}

Linear programming is a method for the optimization of a linear objective function subject to linear inequality constraints. One example of a linear program (LP) is the following:

\begin{equation*}
\begin{matrix}
\displaystyle \text{Minimize} & 2x + 3y - 5z  \\

\textrm{Subject to} & x + z & \leq & 3   \\
& x + y - 3z & \leq & -1  \\
& x,y,z & \geq & 0 & &
\end{matrix}
\end{equation*}

If in an LP all variables are constrained to be integers, the LP is also called an \emph{Integer Program} (IP)\footnote{In fact, we will only use 0-1 valued variables in all IPs throughout this paper.}. All our results will be based on the following two simple observations:
\begin{enumerate}
    \item Many large LPs and IPs can be very efficiently and quickly solved using commercially available LP solvers.
    \item Many conjectures in extremal combinatorics can be phrased as an LP/IP.
\end{enumerate}
Let us now see several examples of how IPs arise naturally in extremal combinatorics. If the reader knows how to code and is new to LP solvers, we would strongly encourage them to implement some of these examples in practice by writing a short script to generate the IPs and solving them with an LP solver. It is difficult to think of a more efficient way to learn the basics of this method than to implement some explicit examples ourselves.

We note that nothing about this method is new. Given the large volume of counterexamples to open conjectures we have found however, we would argue that this technique is not as widely known and used as it could be. Hence we aimed to structure this paper as a guide or tutorial with many examples -- the goal is that the reader will, after working through some of these examples, be able to easily spot when LP solvers could be useful for their own research problems. While most results presented in this paper are counterexamples, this method is also useful for proving results by virtue of knowing whether a statement or an intermediate conjecture is likely to be true or not.

~

\subsection{Example: Sperner's theorem.}\label{subsec:sperner}
As a baby example, consider the problem of finding the largest antichain in the Boolean lattice $2^{[n]}$, for some fixed $n$. A family $\mathcal{F}\subset 2^{[n]}$ is an \emph{antichain} if there are no distinct $A,B\in\mathcal{F}$ with $A\subset B$. This problem can be phrased as an IP as follows. We will have for each set $A\subseteq [n]$ an indicator variable $x_A\in \{0,1\}$, which will take value 1 if $A\in\mathcal{F}$ and 0 otherwise. The fact that $\mathcal{F}$ is an antichain can be encoded by adding for each pair $A,B$ with $A\subsetneq B$ a constraint $x_A + x_B \leq 1$:

\begin{equation*}
\begin{matrix}
\displaystyle \text{Maximize} & \sum_{A\subseteq [n]} x_A \\
\textrm{Subject to} & x_A + x_B & \leq & 1  & ~ \text{for }\forall A,B : A\subsetneq B\subseteq [n] 
\end{matrix}
\end{equation*}

~

Hence this IP has $2^n$ variables and has $3^n-2^n$ linear constraints of the type $x_A+x_B\leq 1$. We will use the LP solver Gurobi~\cite{gurobi} throughout this paper, running on the author's commercially available, average laptop. For $n=10$ it takes around three seconds to produce the correct value of $252=\binom{10}{5}$. If one did not know about Sperner's theorem~\cite{sperner} this would be evidence that the natural guess, saying that the largest layer is the largest antichain, is correct. It is easier to prove a conjecture in practice if one is convinced that it is true. This example may sound trivial, but with only a tiny modification it leads us to new observations.

~

\subsection{Example: a conjecture of Falgas-Ravry.} Given a graph $G$ with $V(G)=[n]$, say a set $A\subseteq [n]$ is \emph{$G$-independent} if the induced subgraph $G[A]$ contains no edges. Denote by $Q(G)$ the poset of all $G$-independent sets under containment, so that if $G$ is the empty graph then $Q(G)$ is the usual Boolean lattice. Denote by $Q^{(r)}(G)$ the set of all $G$-independent sets of size $r$. Denote by $s(G)$ the size of the largest antichain in $Q(G)$. Falgas-Ravry asked, for which $G$ is it true that $s(G)=\max_{0\leq r\leq n} |Q^{(r)}(G)|$? He made the following conjecture. 
\begin{conjecture}[\label{conj:falgas}Falgas-Ravry~\cite{falgas}]
Let $P_n$ denote the path on $n$ vertices. Then
$$s(P_n)=\max_{0\leq r \leq n}|Q^{(r)}(P_n)|.$$
\end{conjecture}
He proved that (one of) the largest antichains must lie in $$Q'(G):=\bigcup_{(n-1)/4 < r < (n+2)/3}Q^{(r)}(P_n),$$
and using this he proved Conjecture~\ref{conj:falgas} for $n\leq 10$ and stated that the $n=11$ case does not look amenable to a pure brute-force search. Using the methods of this paper we can verify his conjecture for all $n\leq 24$.

We phrase this problem as an IP in essentially the same way as we did with Sperner's theorem in Section~\ref{subsec:sperner}. For each $A\in Q'(G)$ we introduce an indicator variable $x_A\in\{0,1\}$, and force our family to be an antichain by adding a constraint $x_A+x_B\leq 1$ whenever $A\subsetneq B$:

\begin{equation*}
\begin{matrix}
\displaystyle \text{Maximize} & \sum_{A\in Q'(P_n)} x_A \\
\textrm{Subject to} & x_A + x_B & \leq & 1  & ~ \text{for }\forall A,B\in Q'(n) : A\subsetneq B 
\end{matrix}
\end{equation*}

One can then solve this IP with an LP solver and easily verify Conjecture~\ref{conj:falgas} for all $n\leq 24$ (and possibly larger $n$, depending on the computational power available). Falgas-Ravry conjectured~\cite{falgas} that the conclusion of Conjecture~\ref{conj:falgas} also holds if one replaces $P_n$ by any vertex-transitive graph. Observe that for any graph $G$ one can set up an IP in the exact same way as above and check whether this conjecture holds.

We can use these ideas to further generalize the questions in this section. Recall the generalization of Sperner's theorem due to Erd\H{o}s~\cite{erdossperner}, which states that in the Boolean lattice the largest $k$-chain free family is given by the $k-1$ largest layers. By using constraints of the form $x_A+x_B+x_C\leq 2$ for each $A\subsetneq B\subsetneq C$ and solving the resulting IP, we find that in $P_{12}$ the largest 3-chain-free family is given by the two largest layers. It would be interesting to know for which graphs the analogue of Erd\H{o}s's theorem holds.

~

\subsection{Example: some conjectures of Bollob\'as--Leader} 
For a family $\F\subset \binom{[n]}{r}$, denote its $d$-neighborhood by $N_d(\F)=\{B\subseteq [n]: |B\Delta F|=d \text{ for some }F\in\F\}$, where $B\Delta F=(B\setminus F) \cup (F\setminus B)$ denotes the symmetric difference.  Bollob\'as and Leader~\cite{bolllead} made several conjectures giving bounds on the maximum size of $N_d(\F)\cap \binom{[n]}{k}$, given fixed integers $d,k,r$ and the size of $\F$. Instead of focusing on only one of their conjectures, we describe how to phrase such problems as IPs in general.
\begin{problem}\label{prob:bolllead}
Given $d,k,r,m,n$. Amongst all families $\F\subset \binom{[n]}{r}$ of size $|\F|=m$, what is the smallest possible size of $N_d(\F)\cap \binom{[n]}{k}$?
\end{problem}

How does one phrase this as an IP? We certainly need to have indicator variables $x_A$ for each $A\in \binom{[n]}{r}$. We ensure that $|\F|=m$ by including the constraint $\sum x_A=m$. Our goal is to minimize the neighborhood. One way to go about this is to include indicator variables $y_B$ for each $B\in\binom{[n]}{k}$. Next, for each $B\in \binom{[n]}{k}$ and $A\in\binom{[n]}{r}$  such that $|A\Delta B|= d$  we include a constraint $x_A \leq y_B$, thus saying that if $A\in\F$ then we must have $B\in N_d(\F)$. The objective then is to minimize $\sum y_B$.

\begin{equation*}
\begin{matrix}
\displaystyle \text{Minimize} & \sum_{B\in \binom{[n]}{k}} y_B \\
\textrm{Subject to} & \sum_{A\in\binom{[n]}{r}}x_A  & = & m   \\
& y_B - x_A & \geq & 0  & ~ \text{for }\forall A\in\binom{[n]}{r}, B\in\binom{[n]}{k} : |A\Delta B|= d
\end{matrix}
\end{equation*}

What if our goal was to find the \emph{largest} possible size of $N_d(\F)\cap \binom{[n]}{k}$ in Problem~\ref{prob:bolllead}? Then we could replace the $y_B-x_A\geq 0$ constraint by $y_B - \sum x_A \leq 0$, where the sum goes over all $A\in \binom{[n]}{r}$ with $|A\Delta B|= d$. 

~

\subsection{Example: coloring graphs and posets} Given a graph $G$, is it 4-colorable? One very simple way to test this using linear programming is as follows: we introduce four families of indicator variables $x_v,y_v,z_v,w_v$ for all $v\in V(G)$ corresponding to the color classes. To ensure that every vertex gets at most one color we include the constraints $x_v+y_v+z_v+w_v\leq 1$ for all $v\in V(G)$. To force all color classes to be independent sets we add for each edge $uv\in E(G)$ the constraint $x_u+x_v\leq 1$, and similarly for the other color classes. The IP can then be written as follows:

\begin{equation*}
\begin{matrix}
\displaystyle \text{Maximize} & \sum_{v\in V(G)} x_v+y_v+z_v+w_v \\
\textrm{Subject to} & x_v+y_v+z_v+w_v  & \leq & 1  & ~ \text{for }\forall v\in V(G) \\
& x_u+x_v & \leq & 1  & ~ \text{for }\forall uv\in E(G)\\
&\ldots\\
& w_u+w_v & \leq & 1  & ~ \text{for }\forall uv\in E(G)
\end{matrix}
\end{equation*}

The graph~$G$ is 4-colorable precisely when the answer to this IP is $|V(G)|$. While there are much better algorithms to check whether a graph is $k$-colorable for fixed $k$, this same idea can be used for seemingly unrelated problems.

Say that a family $\F\subset 2^{[n]}$ is \emph{union-free} if there are no distinct $A,B,C\in\F$ satisfying $A\cup B = C$. In a series of papers, Kleitman~\cite{kleitmanunion} answered a question of Erd\H{o}s by proving that any union-free family has size at most $\binom{n}{\lfloor n/2 \rfloor} + 2^n/n$. Abbott and Hanson~\cite{abbott} raised the following problem: for any integer $n$ let $f(n)$ denote the minimum number of union-free families $\F_1,\F_2\ldots,\F_{f(n)}\subset 2^{[n]}$ such their union is the entire Boolean lattice. What is the value of $f(n)$?

Abbott and Hanson~\cite{abbott} gave an upper bound of $f(n)\leq \lfloor n/2\rfloor + 1$. Erd\H{o}s and Shelah~\cite{erdosshelah} proved that $n/4\leq f(n)$, which was improved first by Aigner--Grieser~\cite{aignergrieser} and then by Aigner--Duffus--Kleitman~\cite{aignerduffuskleitman} to the best current lower bound of 
$\frac{\ln 2}{2}n\leq f(n).$

This problem can be phrased as an IP in a similar way as how we calculated the chromatic number of a graph. The simplest  way is as follows. To check whether $f(n)\leq k$ we can introduce $k$  indicator variables $x_A^{(1)},x_A^{(2)},\ldots,x_A^{(k)}$ for each $A\subseteq [n]$ corresponding to the union-free families. Our goal is then to maximize the sum of all indicator variables, subject to the fact that the families are disjoint, and each family is union-free.

\begin{equation*}
\begin{matrix}
\displaystyle \text{Maximize} & \sum_{i=1}^k\sum_{A\subseteq [n]} x_A^{(i)} \\
\textrm{Subject to} & \sum_{i=1}^k x_A^{(i)}  & \leq & 1  & ~ \text{for }\forall A\subseteq [n] \\
& x_A^{(i)}+x_B^{(i)}+x_C^{(i)} & \leq & 2  & ~ \text{for }\forall i\in [k] \text{ and }A,B,C\subseteq [n]:\\&&&& A\cup B=C, ~ ~ A,B\neq C\\
\end{matrix}
\end{equation*}
The solution of this IP is $2^n$ precisely if $f(n)\leq k$. Solving this IP we can verify that $f(6)=4$, matching the upper bound. We will use this idea to disprove a conjecture related to rainbow matchings later in this paper.

~

\subsection{Example: partitioning a box into proper sub-boxes} Here we describe an application of the methods in the present paper by Bucic, Lidick\'{y}, Long, and the author~\cite{boxes}. A set of the form $\mathcal{A}=A_1\times A_2\times \ldots \times A_d$, where all $A_i$ are finite sets with $|A_i|\geq 2$ is called a \emph{$d$-dimensional discrete box}. A set of the form $\mathcal{B}=B_1\times B_2\times\ldots\times B_d$, where $B_i\subseteq A_i$ for all $i\in [d]$, is a \emph{sub-box} of $\mathcal{A}$. Such a sub-box $\mathcal{B}$ is said to be \emph{proper} if $B_i \neq A_i$ for every $i$. Alon, Bohman, Holzman and Kleitman~\cite{alonbohman} found a beautiful proof answering a question of Kearnes and Kiss, proving that a $d$-dimensional box cannot be partitioned into fewer than $2^d$ proper sub-boxes.

We can verify this for small $d$ and small $A_1,\ldots, A_d$ using an IP as follows. We introduce for each sub-box $\mathcal{B}$ an indicator variable $x_\mathcal{B}$, so that our goal will be to minimize the sum of all these indicator variables. The sub-boxes have to partition the box, hence for each $a\in A_1\times\ldots\times A_d$ we have a constraint $\sum x_\mathcal{B}=1$, where the sum goes over all sub-boxes $\mathcal{B}$ for which $a\in \mathcal{B}$. The sub-boxes we use have to be disjoint, which we can ensure by adding for every pair of intersecting sub-boxes $\mathcal{B},\mathcal{C}$ a constraint $x_\mathcal{B}+x_\mathcal{C}\leq 1$. 

\begin{equation*}
\begin{matrix}
\displaystyle \text{Minimize} & \sum_\mathcal{B} x_\mathcal{B} \\
\textrm{Subject to} & \sum_{\mathcal{B}} \mathbbm{1}_{a\in \mathcal{B}}\cdot x_\mathcal{B}  & = & 1  & ~ \text{for }\forall a\in \mathcal{A} \\
& x_\mathcal{B}+x_\mathcal{C} & \leq & 1  & ~ \text{for }\forall \mathcal{B},\mathcal{C}: \mathcal{B}\cap \mathcal{C}\neq \emptyset 
\end{matrix}
\end{equation*}

What happens if $|A_i|$ is odd for each $i$ and we want to partition only with sub-boxes $B_1\times \ldots \times B_d$ with $|B_i|$ odd for all $i$? There is a natural partition into $3^d$ such boxes, and Leader, Mili\'{c}evi\'{c} and Tan~\cite{leadermili} asked whether this is best possible.  By phrasing this question as an IP in essentially the same way as above and solving the resulting IP, in~\cite{boxes} we found that $[5]\times [5]\times [5]$ could be partitioned into 25 boxes. This then led us to proving that in general $2.93^d$ boxes are enough, answering the question of Leader, Mili\'{c}evi\'{c} and Tan.

~

\subsection{Example: Increasing the distance in the hypercube}

Let $Q_n$ be the hypercube, that is, the graph with vertex set $2^{[n]}$ and sets $A,B\subset [n]$ connected by an edge if $|A\Delta B|=1$. Two sets $A,B\in V(Q_n)$ are called \emph{antipodal} if $A=[n]\setminus B$. Yuzvinsky asked the following question:
\begin{question}[\label{ques:yuz}Yuzvinsky]
How many vertices must be
removed from the $n$-cube in order that no connected component of the remainder
contains an antipodal pair of vertices?
\end{question}
Kleitman answered Question~\ref{ques:yuz} in a strong way:
\begin{theorem}[\label{thm:kleitman1}Kleitman~\cite{yuzvinsky}]
Suppose removal of the vertices of the set $X$ leaves the n-cube with no connected component of size exceeding $2^{n-1}$. Then $X$ must contain at least $\binom{n}{\lfloor n/2\rfloor}$ vertices.
\end{theorem}

Note that Theorem~\ref{thm:kleitman1} implies that the answer to Question~\ref{ques:yuz} is $\binom{n}{n/2}$ -- indeed, if removing $X$ has no connected component that contains an antipodal pair of vertices, then in particular no connected component has size exceeding $2^{n-1}$. Frankl asked the following follow-up question:
\begin{question}[Frankl~\cite{yuzvinsky}]
How few vertices can one remove from
the $n$-cube so that the distance between antipodal pairs in the $n$-cube becomes at
least $n + 2$ using only remaining edges? 
\end{question}

(Note that the distance cannot be $n+1$ because of parity.) Certainly the answer is not greater than $\binom{n}{\lfloor n/2 \rfloor}$. We can phrase this as an IP as follows. Observe that the condition that removing a set $S$ of vertices increases the distance between any two antipodal pairs is equivalent to saying that $S$ intersects every geodesic on $n+1$ vertices. Here a geodesic on $n+1$ vertices given by a starting vertex $A_0\in \{0,1\}^n$ and a permutation $\pi\in S_n$ is the sequence of vertices $A_0,A_1,\ldots,A_{n}$ where $A_{i}$ is obtained from $A_{i-1}$ by flipping the $\pi(i)$-th digit. Using indicator variables representing the elements of $S$, we get he following IP:

\begin{equation*}
\begin{matrix}
\displaystyle \text{Minimize} & \sum_{A\subseteq [n]} x_A \\
\textrm{Subject to} & \sum_{A\in \mathcal{G}} x_A  & \geq & 1  & ~ \text{for  every geodesic }\mathcal{G} \\
\end{matrix}
\end{equation*}

Solving this IP we find that the answer is $\binom{n}{\lfloor n/2 \rfloor}$ for $n\leq 7$, and it is plausible that removing the middle layer is the best one can do for all $n$.

~

The above examples illustrate how various problems in extremal combinatorics can be phrased as IPs. Typically this is not the best way to approach any one of these problems. The main advantage of this method is that it provides a very quick way to perform a sanity check on a wide variety of conjectures.  In practice it only takes a few minutes to write a program to generate the IP, and only a few seconds to solve it with the LP solver for small values of the parameters and verify that there are no small counterexamples.

Let us now turn to the main part of the paper, where we use these methods to resolve some open conjectures and questions in combinatorics.

\section{Main results}~\label{sec:results}

\subsection{Antichains of fixed diameter}\label{subsec:anti}

Define the \emph{diameter} $\mathrm{diam}(\F)$ of a family $\F\subset 2^{[n]}$ as $\mathrm{diam}(\F)=\max_{A,B\in\F} \{|(A\setminus B)\cup (B\setminus A)|\}$. Frankl~\cite{frankldiam} considered the problem of determining the largest size of an antichain in $2^{[n]}$ of diameter at most $d$. 

We phrase this problem as an integer program as follows. We introduce for each set $A\subset [n]$ a 0-1 valued variable $x_A$ that indicates whether $A\in\F$. We can force $\F$ to be an antichain by adding, for each comparable pair $A\subsetneq B$ a linear constraint $x_A+x_B\leq 1$. Next, to ensure that the solution has diameter at most $d$, for each pair $A,B$ with $|(A\setminus B)\cup (B\setminus A)|>d$ we add a restriction $x_A+x_B\leq 1$. Frankl~\cite{frankldiam} made the following conjecture:

\begin{conjecture}[\label{conj:frankldiam1}\cite{frankldiam}]Let $n, d$ be positive integers, $n > d$. Suppose that $\F \subset 2^{[n]}$ is an antichain with diameter $\mathrm{diam}(\F) \leq d$. Then
$$|\F|\leq \binom{n}{\lfloor d/2\rfloor}.$$
\end{conjecture}
Frankl proved~\cite{frankldiam} Conjecture~\ref{conj:frankldiam1} for $n \geq 6(r+1)^2$. Using an LP solver we attempted to find a counterexample to this conjecture, but according to the computer it holds for the values $(n,d)=(10,3), (8,5), (8,7)$.

Frankl also made a similar conjecture for $k$-chain free families. A $k$-chain is a collection of $k$ sets $A_1\subsetneq A_2\subsetneq \ldots \subsetneq A_k$ totally ordered under inclusion. Observe that, similarly to being an antichain, the property of being $k$-chain-free can be captured by an integer program by adding for each $k$-chain $A\subsetneq B\subsetneq\ldots C$ a linear constraint $x_A+x_B+\ldots +x_C\leq k-1$.
\begin{conjecture}[\label{conj:frankldiam2}\cite{frankldiam}]
Let $n, d,\ell$ be positive integers, $n > d\geq \ell$. Suppose that $\F \subset 2^{[n]}$ is $\ell + 1$-chain-free with diameter $\mathrm{diam}(\F) \leq d$. Then setting $s=\min\{\ell - 1,\lfloor d/2 \rfloor\}$ one has
$$|\F|\leq \sum_{\lfloor d/2 \rfloor \geq i \geq \lfloor d/2 \rfloor - s}\binom{n}{i}.$$
\end{conjecture}
Frankl noted that the special case $s=\lfloor d/2 \rfloor$ follows directly from Kleitman's diameter theorem~\cite{kleitmandiam}. Frankl also wrote~\cite{frankldiam} that it follows from the methods of his paper that Conjecture~\ref{conj:frankldiam2} holds for $n$ large enough. This claim is incorrect, as shown below.

By solving the IP directly, we obtain a counterexample for $n=6, d=5, \ell = 2$. For these parameters the bound given in Conjecture~\ref{conj:frankldiam2} is $\binom{n}{2}+\binom{n}{1} = 21$, but in fact a family of size 26 exists: $$\F=\binom{[6]}{2} \cup
\left\{A\in \binom{[6]}{3}: 1\in A\right\}\cup\{23456\}.$$ 
Given this example given by the computer, it is easy to realize that the family given by the entire 2-layer together with a star on the third layer has diameter 5 and is 3-chain-free for any $n$. (Here a star means all sets containing a fixed element.) It has size $\binom{n}{2}+\binom{n-1}{2}=(n-1)^2$ which is bigger than the construction implied by Conjecture~\ref{conj:frankldiam2} for all $n\geq 6$, but as the example shows it is not best possible for $n=6$. Surprisingly, according to the LP solver this bound is tight for $n=7,8$. It is plausible that this is tight for all $n\geq 7$. 

For $n=8, d=7, \ell = 2$ the LP solver gives that the best solution is to take stars centered at $\{1\}$ in layers 2 and 4, together with all sets avoiding $\{1\}$ on layers 3 and 5, giving a family of size $\binom{7}{1}+\binom{7}{3} + \binom{7}{3} + \binom{7}{5}=98$ beating the value 84 given by Conjecture~\ref{conj:frankldiam2}. For $n=9, d=7, \ell=2$ a construction of size 141 is given by taking the full third layer, a star centered on $\{1\}$ on the fourth layer and the single set $\{2,3,4,\ldots,9\}$. It is beyond our computational limits to see if this is best possible for $n=9, d=7,\ell=2$. 

In general for $d$ odd one can take the entire $\lfloor d/2\rfloor$ layer together with a star on the layer above to get a construction of size $\binom{n}{\lfloor d/2 \rfloor} + \binom{n-1}{\lfloor d/2 \rfloor}$ and it is plausible that for $n\geq n_0(d)$ this is the best one can do to avoid 3-chains. 

~

\subsection{Diversity of set systems}\label{subsec:int}

Let $n$ and $k$ be positive integers with $n>2k$, and let $\F\subset\binom{[n]}{k}$ be an intersecting family. Denote the maximum degree of $\F$ by $\Delta(\F)$, so that $\Delta(\F) = \max_i |\F(i)|$ where $\F(i)=\{F\in\F:i\in F\}$. We define the \emph{diversity} $\rho(\F)$ of $\F$ by $\rho(\F)=|\F|-\Delta(\F)$. Observe that if $\F$ is a star, i.e.~all sets in $\F$ contain a fixed element, then $\rho(\F)=0$ and otherwise we have $\rho(\F)\geq 1$. Lemons and Palmer~\cite{lempal} proved that for $n>6k^3$ we have $\rho(\F)\leq \binom{n-3}{k-2}$. The following conjecture was made by Frankl~(\cite{frankltoku}, p.~214):

\begin{conjecture}[\label{conj:frankldiv}\cite{frankltoku}]
Suppose that $n>3(k-1)$ and $\F\subset \binom{[n]}{k}$ is intersecting. Then
$$\rho(\F)\leq \binom{n-3}{k-2}.$$
\end{conjecture}
The example $\F=\{F\in \binom{[n]}{k}: |F\cap [5]|\geq 3 \}$ shows that this conjecture fails for $n\leq 3(k-1)$.  This conjecture was proved by Frankl~\cite{frankldiam} for $n\geq 6k^2$ and recently by Kupavskii~\cite{kupadiv} for $n>Ck$ for some large constant $C$. 

One can phrase this problem as an IP as follows. Fix some $n$ and $k$ and introduce an indicator variable $x_A$ for each $A\in \binom{[n]}{k}$ as usual. To ensure that $\F$ is intersecting, for any distinct pair of sets $A,B$ add a constraint $x_A+x_B\leq 1$. Next we ensure that the diversity of $\F$ is attained at the element 1. That is, for each $i\neq 1$ we add a constraint
$$\sum_{A\subset [n], 1\not\in A} x_A - \sum_{A\subset [n], i\not\in A} x_A \leq 0.$$

Then in order to find a family $\F$ in $\binom{[n]}{k}$ with the largest possible diversity, we need to maximize $\sum_{A\subset [n], 1\not\in A} x_A$ subject to these constraints. Solving this IP yields a very small counterexample with $n=7, k=3$: $$\F=\{235,236,246,345,456,124,125,134,136,156\}.$$ Then $\F$ is intersecting, as all sets are 3-subsets of $[6]$ and we picked exactly one from each complementary pair. Every vertex (except for the isolated one) has degree exactly $5$ and hence $\rho(\F) = |\F|-5=5$, but $\binom{n-3}{k-2} = \binom{4}{1}=4$. We note that this also disproves a stronger conjecture of Frankl on the same page of~\cite{frankltoku}.

While writing this paper, it was brought to our attention by Andrey Kupavskii that Huang~\cite{huang} had independently, but strictly before us, disproved Conjecture~\ref{conj:frankldiv}. In his paper he makes two new conjectures on the maximum diversity of families, and we shall disprove both.

\begin{conjecture}[\cite{huang}\label{conj:huang1}]
For $n=2k+1$, suppose $\F\subset 2^{[n]}$ is intersecting. Then
$$\rho(\F)\leq \rho(\mathcal{Q}_k)=\sum_{i=k+1}^{2k}\binom{2k}{i}.$$
\end{conjecture}
Here $\mathcal{Q}_k=\{A: A\subset [2k+1], |A|\geq k+1\}$. The second conjecture concerns the case $n=2k$.
\begin{conjecture}[\cite{huang}\label{conj:huang2}]
For $n=2k$, suppose $\F\subset 2^{[n]}$ is intersecting. If $k$ is not a power of 2, then
$$\rho(\F)\leq \frac12 \binom{2k-1}{k-1} + \sum_{i=k+1}^{2k-1}\binom{2k-1}{i};$$
and if $k$ is a power of 2, then
$$\rho(\F)\leq \frac12 \left(\binom{2k-1}{k-1}-1\right) + \sum_{i=k+1}^{2k-1}\binom{2k-1}{i}.$$
\end{conjecture}
Huang checked~\cite{huang} Conjectures~\ref{conj:huang1} and~\ref{conj:huang2} for $n\leq 6$ using a computer. Using the methods of the present paper however, it is straightforward to check these conjectures for larger values of $n$.

The set-up for the integer programs is essentially the same as the set-up for Conjecture~\ref{conj:frankldiv}. For $n=7$ and $n=9$ it takes seconds of run-time to find counterexamples to Conjecture~\ref{conj:huang1}, and for $n=10$ it took less than a minute of run-time to find a counterexample to Conjecture~\ref{conj:huang2}. Perhaps unsurprisingly all counterexample families $\F$ we found are up-closed, meaning that if $A\in \F$ and $A\subset B$ then $B\in\F$.
For $n=7$, a construction with diversity of 23 which beats the value of 22 in the conjecture is:

\begin{equation*}
\begin{split}
    \F =\{&123, 146, 157, 247, 256, 345, 367, 1234, 1235, 1236, 1237, 1246, 1247, 1256, \\ &1257, 1345, 1346, 1357, 1367, 1456, 1457, 1467, 1567, 2345, 2347, 2356, \\ &2367, 2456, 2457, 2467, 2567, 3456, 3457, 3467, 3567\} \cup\binom{[7]}{5}\cup\binom{[7]}{6}.
\end{split}
\end{equation*}

The construction above can be thought of as follows: start with the Fano plane $\F_7 = \{123, 146, 157, 247, 256, 345, 367\}$ and let $\F$ consist of all those  $A\subseteq[n]$ for which there exists some $B\in\F_7$ such that $B\subseteq A$.

For $n=9$ the correct answer according to the LP solver is 97, beating the conjectured value of 93. The example is constructed from a particular 4-uniform, regular family $\F_9$ together with all supersets of its elements, as before. Here $\F_9$ is as follows:

\begin{equation*}
\begin{split}
    \F_9 =\{&1234, 1236, 1239, 1245, 1248, 1289, 1347, 1367, 1368, 1457, 1459, 1567, \\ &1578, 1589, 1679, 1689, 2356, 2359, 2379, 2456, 2468, 2478, 2567, 2578, \\ &2679, 2789, 3459, 3468, 3478, 3479, 3568, 3578, 3589, 4569, 4679, 4689\}
\end{split}
\end{equation*}

For $n=10$ we have a construction of size $197$, beating the value of $193$. It is presented in the Appendix.

~

\subsection{Multipartite intersecting families}\label{subsec:kat}

A celebrated theorem of Erd\H{o}s--Ko--Rado is the following:
\begin{theorem}[Erd\H{o}s--Ko--Rado~\cite{erdoskorado}\label{thm:erdoskorado}]
Given integers $n,k$ with $k\leq n/2$, if $\F\subset \binom{[n]}{k}$ is intersecting then 
$$|\F|\leq \binom{n-1}{k-1}.$$
\end{theorem}
Equality in Theorem~\ref{thm:erdoskorado} is attained by the family of all $k$-sets containing a fixed element, which we refer to as a \emph{trivially intersecting family} or a \emph{star}. Hilton and Milner~\cite{hiltonmilner} found the largest intersecting, but not trivially intersecting family:
\begin{theorem}[Hilton--Milner~\cite{hiltonmilner}\label{thm:hiltonmilner}]
If $2k\leq n$ and $\F$ is an intersecting but not trivially intersecting family in $\binom{[n]}{k}$ then
$$|\F|\leq 1+\binom{n-1}{k-1} - \binom{n-k-1}{k-1}.$$
\end{theorem}
Let $X_1$ and $X_2$ be disjoint sets of size $n_1$ and $n_2$ respectively, and denote by $\binom{X_1,X_2}{k,\ell}$ the family of all sets $S\subset X_1\dot{\cup} X_2$ with $|S\cap X_1|=k$ and $|S\cap X_2|=\ell$. Frankl~\cite{franklbip} and Katona~\cite{katona} considered intersecting families in $\binom{X_1,X_2}{k,\ell}$. As before, a family $\F\subset\binom{X_1,X_2}{k,\ell}$ is trivially intersecting if all elements of $\F$ contain a fixed element. Katona~\cite{katona} observed that if $x\in X_1$ is an arbitrary element and $K\subset X_1\setminus \{x\}$ is a set of size $k$ then the family
$$\F=\left\{F\in\binom{X_1,X_2}{k,\ell}: x\in F, F\cap K \neq \emptyset\right\}\cup\{K\}$$
is intersecting but not trivially intersecting. Motivated by this, Katona~\cite{katona} made the following conjecture.
\begin{conjecture}[Katona~\cite{katona}\label{conj:katona1}]
If $\F$ is an intersecting but not trivially intersecting subfamily of $\binom{X_1,X_2}{k,\ell}$ then
\begin{equation*}
    \begin{split}
        |\F| \leq  \max\bigg\{& \left(1+\binom{n_1-1}{k-1}-\binom{n_1-k-1}{k-1}\right)\binom{n_2}{\ell}, \\& \binom{n_1}{k} \left(1+\binom{n_2-1}{\ell-1}-\binom{n_2-\ell-1}{\ell-1}\right) \bigg\}.
    \end{split}
\end{equation*}
\end{conjecture}
We assume that the conditions $2k\leq n_1$ and $2\ell \leq n_2$ are implicitly implied in Conjecture~\ref{conj:katona1}. Katona also made a conjecture on \emph{two-sided intersecting} families, i.e.~families $\F$ for which there exist members $F_{11},F_{12},F_{21},F_{22}\in\F$ such that $F_{11}\cap F_{12}\cap X_1=\emptyset$ and $F_{21}\cap F_{22}\cap X_2=\emptyset$. 
\begin{conjecture}[\cite{katona}\label{conj:katona2}]
If $\F$ is a two-sided intersecting subfamily of $\binom{X_1,X_2}{k,\ell}$ then
\begin{equation*}
    \begin{split}
        |\F|\leq \max\bigg\{&\left(\binom{n_2-1}{\ell-1}-\binom{n_2-\ell-1}{\ell-1}\right)\binom{n_1}{k}+1+\binom{n_1}{k}-\binom{n_1-k}{k},\\
        &\left(\binom{n_1-1}{k-1}-\binom{n_1-k-1}{k-1}\right)\binom{n_2}{\ell}+1+\binom{n_2}{\ell}-\binom{n_2-\ell}{\ell}\bigg\}.
    \end{split}
\end{equation*}
\end{conjecture}
Once again we assume the conditions $2k\leq n_1$ and $2\ell\leq n_2$ are implicit. Let us now try to disprove both conjectures~\ref{conj:katona1} and~\ref{conj:katona2}. We phrase them as IPs as follows. We fix some $X_1,X_2,k,\ell$ and for each element $F$ of $\binom{X_1,X_2}{k,\ell}$ we introduce an indicator variable $x_F$. We force $\F$ to be intersecting by adding for each disjoint pair of sets $F,G$ a constraint $x_F+x_G\leq 1$. For Conjecture~\ref{conj:katona1} we ensure that $\F$ is not trivially intersecting by adding for each $x\in X_1\cup X_2$ a constraint
$$\sum_{x\not\in F}x_F\geq 1.$$
For Conjecture~\ref{conj:katona2} we force the two-sided intersecting property in a similar fashion. We pick two disjoint $k$-sets $L_1,L_2\subset X_1$ and two disjoint $\ell$-sets $R_1,R_2\subset X_2$. Then we add for each $S\in\{L_1,L_2,R_1,R_2\}$ the constraint
$$\sum_{S\subset F}x_F\geq 1.$$ Solving the IP directly yields a counterexample for both conjectures for $n_1=n_2=5$, $k=\ell=2$ in less than a second. Let $X=\{x_1,x_2,\ldots,x_5\}$ and $Y=\{y_1,y_2,\ldots,y_5\}$. Let $\F$ be the following family:
\begin{equation*}
    \begin{split}
        \F=&\left\{\{x_1,x_2\}\cup F: F\in\binom{Y}{2}, F\cap \{y_1,y_2\}\neq \emptyset\right\}\\\cup& \left\{\{x_1,x_3\}\cup F: F\in\binom{Y}{2}, F\cap \{y_1,y_2\}\neq \emptyset\right\}\\\cup& \left\{\{x_1,x_4\}\cup F:F\in\binom{Y}{2}\right\} \cup \left\{\{x_1,x_5\}\cup F:F\in\binom{Y}{2}\right\}\\\cup&\{\{x_4,x_5,y_1,y_2\}\}
    \end{split}
\end{equation*}
The size of $\F$ is 35, while the constructions in conjectures~\ref{conj:katona1} and~\ref{conj:katona2} have sizes 30 and 28 respectively. This construction generalizes for $\binom{X,Y}{2,2}$. For simplicity assume $|X|=|Y|\geq 5$.
\begin{proposition}\label{prop:katonaconstr}
Let $X=\{x_1,x_2,\ldots,x_m\}$ and $Y=\{y_1,y_2,\ldots,y_m\}$ be two disjoint sets of size $m\geq 5$. Then there is a $\F\subset\binom{X,Y}{2,2}$ that is two-sided intersecting, with $|F|\geq 3m^2-10m+10$.
\end{proposition}
\begin{proof}
Let $\F$ be defined as follows.
\begin{equation*}
    \begin{split}
        \F=&\left\{G\cup F: G\in\binom{X}{2}, F\in\binom{Y}{2}, x_1\in G, F\cap \{y_1,y_2\}\neq \emptyset\right\}\\\cup& \left\{\{x_1,x_2\}\cup F:F\in\binom{Y}{2}\right\} \cup \left\{\{x_1,x_3\}\cup F:F\in\binom{Y}{2}\right\}\\\cup&\{\{x_2,x_3,y_1,y_2\}\}
    \end{split}
\end{equation*}
The size of $\F$ is then given by 
$$|\F|=(m-3)\left(\binom{m}{2}-\binom{m-2}{2}\right)+2\binom{m}{2}+1=3m^2-10m+10.$$
\end{proof}
We note that according to the LP solver, the construction in Proposition~\ref{prop:katonaconstr} is in fact the largest non-trivially intersecting (but not necessarily two-sided intersecting) family for $m=5,6$.

For $(n_1,n_2,k,\ell)=(7,7,3,3)$ we find a two-sided intersecting family of size 514, beating the values of 455 and 452 in conjectures~\ref{conj:katona1} and~\ref{conj:katona2} respectively.  Based on generalizing the construction given by the LP solver, we have the following bound. 
\begin{proposition}\label{prop:katonaconstr2}
Let $k,m$ be integers with $2k\leq m$. Let $X=\{x_1,x_2,\ldots,x_m\}$ and $Y=\{y_1,y_2,\ldots,y_m\}$ be two disjoint sets of size $m$. Then there is a $\F\subset\binom{X,Y}{k,k}$ that is two-sided intersecting, with 
$$|\F|\geq \left(\binom{m-1}{k-1}-\binom{m-k-1}{k-1}\right)\binom{m}{k}+\binom{m-k-1}{k-1}\left(\binom{m}{k}-\binom{m-k}{k}\right) +1.$$
\end{proposition}
\begin{proof}
Let $K_1=\{x_2,x_3,\ldots,x_{k+1}\}$, $K_2=\{y_1,y_2,\ldots,y_{k}\}$ and define the family as
\begin{equation*}
    \begin{split}
        \F=&\left\{F_1\cup F_2: F_1\in\binom{X}{k}, F_2\in\binom{Y}{k},  x_1\in F_1, F_1\cap K_1\neq \emptyset, \right\}\\\cup& \left\{F_1\cup F_2: F_1\in\binom{X}{k}, F_2\in\binom{Y}{k},  x_1\in F_1, F_1\cap K_1=\emptyset,  F_2\cap K_2\neq\emptyset \right\}\\ \cup& \{K_1\cup K_2\}.
    \end{split}
\end{equation*}
\end{proof}

\noindent
It would be interesting to see whether this construction is best possible. 

Let us now turn to another related conjecture, by Frankl--Han--Huang--Zhao~\cite{franklhanhuangzhao}. We say that a family \emph{has the EKR property} if its largest intersecting subfamily is trivially intersecting.
\begin{conjecture}[\label{conj:franklhan}\cite{franklhanhuangzhao}]Suppose $n=n_1+\ldots +n_d$ and $k\geq k_1+\ldots +k_d$, where $n_i > k_i\geq 0$ are integers. Let $X_1\cup \ldots \cup X_d$ be a partition of $[n]$ with $|X_i|=n_i$, and $$\mathcal{H}=\left\{F\subseteq \binom{[n]}{k}: |F\cap X_i|\geq k_i \text{ for } i=1,\ldots,d\right\}.$$
If $n_i\geq 2k_i$ for all $i$ and $n_i>k-\sum_{j=1}^d k_j + k_i$ for all but at most one $i\in [d]$ such that $k_i>0$, then $\mathcal{H}$ has the EKR property.
\end{conjecture}
We first observe that if e.g.~$d=2, n_1=3, n_2=4, k_1=1, k_2=2, k=4$ then all conditions of the conjecture are satisfied but $\mathcal{H}$ itself is (non-trivially) intersecting, and so Conjecture~\ref{conj:franklhan} cannot be true. In particular for this set of parameters $|\mathcal{H}|=30$ but the largest trivially intersecting subfamily of $\mathcal{H}$ has size 18. We will thus assume that the $n\geq 2k$ condition was intended to be a part of the statement of Conjecture~\ref{conj:franklhan}.

We phrase this problem as an IP in much the same way as before. Fix some values for the parameters, and introduce indicator variables $x_F$ for each $F\in\mathcal{H}$. Then add constraints $x_F+x_H\leq 1$ for each disjoint $F,H\in\mathcal{H}$. Solving the LP yields counterexamples for several sets of parameters. The smallest we could find is for the values
$d = 2, n_1=n_2=4, k_1=2, k_2=1, k=4$, so that
$$\mathcal{H}=\left\{F\subseteq \binom{[8]}{4}: |F\cap \{1,2,3,4\}|\geq 2, ~ |F\cap \{5,6,7,8\}|\geq 1\right\}.$$
The largest trivially intersecting family in $\mathcal{H}$ has size 30, but its largest intersecting subfamily has order 34:

\begin{equation*}
    \begin{split}
        \F =\{&1235, 1236, 1237, 1238, 1245, 1246, 1247, 1248, 1256, 1267, 1268, 1278, \\ &1345, 1346, 1347, 1348, 1358, 1368, 1378, 1467, 1468, 2345, 2346, 2347, \\ &2348, 2356, 2367, 2368, 2378, 2458, 2468, 2478, 3467, 3468\}
    \end{split}
\end{equation*}

\subsection{A forbidden trace problem}\label{subsec:anstee}

To introduce the definition of a \emph{forbidden configuration}, we will use the language of matrix theory and identify set systems with their adjacency matrix. An $m \times n$ simple matrix (i.e.~with no repeated columns) $A$ with all entries in $\{0,1\}$ can be thought of as a family $\mathcal{A}$ of $n$ subsets of $[m]$:  the rows index the elements of the ground sets and the columns index the subsets. So the number of columns of $A$ is equal to $|\mathcal{A}|$. 

Now let $F$ be a $k\times \ell$ matrix with all entries in $\{0,1\}$. We say that a matrix $A$ \emph{has a configuration $F$} if a submatrix of $A$ is a row and column permutation of $F$  (this is sometimes called \emph{trace} in the language of sets).

Many classical problems in extremal set theory can be phrased as problems about forbidden configurations. One standard example is bounding the size of a family of VC dimension at most $k$. We say that a family $\mathcal{A}\subset 2^{[n]}$ has VC dimension at least $k$ if there exists a set $S\subset [n]$ of size $|S|=k$ such that $|A\cap S : A\in\mathcal{A}|=2^k$. Hence a family $\mathcal{A}$ has VC dimension less than 3 if and only if the corresponding matrix $A$ does not have configuration $F_3$, where $F_3$ is the matrix
$$F_3=\begin{bmatrix}
    0 & 1 & 0 & 0 & 1 & 1 & 0 & 1 \\
    0 & 0 & 1 & 0 & 1 & 0 & 1 & 1 \\
    0 & 0 & 0 & 1 & 0 & 1 & 1 & 1
\end{bmatrix}$$

Denote by $\mathrm{forb}(m,F)$ the maximum number of columns in a matrix $A$ without a configuration $F$. So e.g.~by a classical theorem of Sauer--Shelah we have $\mathrm{forb}(m,F_3)=\binom{m}{0}+\binom{m}{1}+\binom{m}{2}$. We refer the reader to the excellent survey of Anstee~\cite{anstee} on more background on forbidden configuration problems.

Steiner triple systems are one of the most classical objects studied in combinatorial design theory, dating back to Kirkman~\cite{kirkman}. We say a family  of $3$-element subsets, called \emph{blocks}, of an $n$-element set $X$ is a triple system of multiplicity $\lambda$ if any pair of distinct elements of $X$ are contained in precisely $\lambda$ blocks. Anstee raised the following problem, which we will disprove:

\begin{problem}[\cite{anstee}\label{prob:anstee}]
Show that for those $m$ for which a triple system of multiplicity 2 exists,
$$\text{forb}\left(m,\begin{bmatrix}
  1 & 1 & 1 & 1\\
1 & 1 & 1 & 1\\
1 & 0 & 0 & 0
\end{bmatrix}\right) = \frac{5}{3}\binom{m}{2}+\binom{m}{1}+\binom{m}{0}+\binom{m}{m}.$$
\end{problem}

Denote the forbidden matrix by $A$. As a triple system of multiplicity one has order $\frac13\binom{m}{2}$, it is our guess that the intended construction achieving the bound on the right hand side is $\binom{[m]}{\leq 2}\cup \{[m]\} $ together with a triple system of multiplicity two. However, this construction does in fact contain the forbidden matrix $A$ as a configuration, but removing the single set $\{[m]\}$ would fix the issue. Nevertheless, we are able to find a construction that is larger than the value on the right hand side.

We can phrase this problem as an IP as follows. We introduce for each set $S\subseteq [n]$ a 0-1 valued indicator variable $x_S$. For any four distinct sets $A,B,C,D$ if there exist three elements of the ground set such that the trace of $A,B,C,D$ on these three elements would give the forbidden matrix $A$, then we add a constraint $x_A+x_B+x_C+x_D\leq 3$. The objective is then to maximize the sum of all variables. 

We begin by noting that there exists a triple system of multiplicity two for $m=4,6,7,9$, see~\cite{designhandbook}. Denote by $S_2(m)$ a triple system of multiplicity two and order $m$, for those $m$ where it exists.  Next we observe that we may assume the family $\F$ contains all sets of size 0 or 1 as these do not affect containment of our forbidden configuration, hence we may restrict our search space on $\binom{[m]}{\geq 2}$. Solving the IP directly for $m=6$ we find that statement of Problem~\ref{prob:anstee} is false, the correct answer is 25 rather than 26 -- here 25 is given by the natural construction $\F = \binom{[6]}{\leq 2}\cup S_2(6)$. 

For $m=9$ solving the IP was infeasible with the author's laptop. By making the heuristic assumption that the optimal family should contain all sets of size at most two and restricting the search to $\binom{[m]}{3}\cup\binom{[m]}{4}$ we find a construction of size $71$ within three minutes -- this matches the bound given in Problem~\ref{prob:anstee}, and hence beats the natural construction of $\F=\binom{[m]}{\leq 2}\cup S_2(m)$ by one! The construction given by the LP solver is as follows: take all sets of size at most two, together with a triple system of multiplicity two, which contains the triples $\{123,124,134,234\}$, and add the single set $\{1234\}$. Such a triple system indeed exists, see e.g.~\cite{designhandbook}. Hence if we could find a triple system of multiplicity two of some higher order, that contains the triples $\{123,124,134,234, 567,568,578,678\}$ then we could add two 4-sets and beat the bound in Problem~\ref{prob:anstee}. We will need the following theorem of Colbourn--Hamm--Lindner--Rodger:
\begin{theorem}[Colbourn--Hamm--Lindner--Rodger~\cite{colbournembed}\label{thm:tripleembed}]
A partial triple system of order $m$ and multiplicity $\lambda$ can be embedded in a triple system of multiplicity $\lambda$ and order at most $4(3\lambda/2 + 1)m + 1$. 
\end{theorem}
\begin{proposition}\label{prop:tripleembed}
For every $k\geq 1$ there exists an $m\leq 64k+1$ such that $$\text{forb}\left(m,\begin{bmatrix}
  1 & 1 & 1 & 1\\
1 & 1 & 1 & 1\\
1 & 0 & 0 & 0
\end{bmatrix}\right) \geq \frac{5}{3}\binom{m}{2}+\binom{m}{1}+\binom{m}{0}+k.$$
\end{proposition}
\begin{proof}
Construct a partial triple system of multiplicity 2 by taking, for all $0\leq i \leq k-1$, the triples $\{4i+1,4i+2,4i+3\}$, $\{4i+1,4i+2,4i+4\}$, $\{4i+1,4i+3,4i+4\}$ and $\{4i+2,4i+3,4i+4\}$. These $4k$ triples form a partial triple system of multiplicity 2 and order $4k$. By Theorem~\ref{thm:tripleembed} these triples are contained in some triple system $\F$ with $\lambda = 2$ and order $m$, with $m\leq 16\cdot 4k+1$. Adding to $\F$ all sets of size two or less and the 4-sets $\{4i+1,4i+2,4i+3,4i+4\}$ for all $0\leq i \leq k-1$ we obtain a family of the correct size, which does not contain the forbidden configuration given by the matrix $A$.
\end{proof}

We observe that the bound in Proposition~\ref{prop:tripleembed} is not sharp, in particular any two of the added 4-sets could be allowed to intersect in one element. Indeed, if four sets are witnesses for the configuration $A$ then any two of the four sets intersect in at least two elements. This leads us to an even stronger bound, giving an improvement in the leading coefficient.
\begin{proposition}
For every sufficiently large $m$ with $m\equiv 1,4\text{ (mod 12)}$ we have
$$\text{forb}\left(m,\begin{bmatrix}
  1 & 1 & 1 & 1\\
1 & 1 & 1 & 1\\
1 & 0 & 0 & 0
\end{bmatrix}\right) \geq \frac{11}{6}\binom{m}{2}+\binom{m}{1}+\binom{m}{0}.$$
\end{proposition}
\begin{proof}
Given a sufficiently large integer $m\equiv 1,4 \text{ (mod 12)}$, by a theorem of Wilson~\cite{wilson} there exists a family $\mathcal{S}$ of 4-sets in $[m]$ such that any pair $x,y\in [m]$ of distinct elements are covered by precisely one set in $\mathcal{S}$. Note that $|\mathcal{S}|=\binom{m}{2}/6$. Construct a family $\F'$ of 3-sets by including, for each set $S\in\mathcal{S}$, all four 3-subsets of $S$. Note that since any two sets in $\mathcal{S}$ intersect in at most one element, we have $|\F'|=\frac23\binom{m}{2}$ and in fact $\F'$ is a triple system of order $m$ and multiplicity two. Then the family $$\F=\binom{[m]}{\leq 2}\cup \F'\cup \mathcal{S}$$
does not contain the forbidden configuration and has the correct size.
\end{proof}

~

\subsection{A Hoffman-type eigenvalue bound on regular set systems}\label{subsec:hoffman}

We say that a family $\F\subset 2^{[n]}$ is \emph{$s$-subset-regular} if every set of size $s$ lies in the same number of elements of $\F$. Ihringer and Kupavskii~\cite{ihrkup} proved the following Hoffman-type eigenvalue upper bound on such regular families:
\begin{theorem}[Ihringer--Kupavskii~\cite{ihrkup}\label{thm:ihrkup}]
Fix odd $s\geq 1$. An $s$-subset-regular $k$-uniform intersecting family $\F$ on $[n]$ satisfies
$$|\F|\leq \frac{\binom{n}{k}}{1+\frac{\binom{n-k}{k}}{\binom{n-k-s-2}{k-s-2}}}.$$
\end{theorem}
They proved~\cite{ihrkup} that equality in Theorem~\ref{thm:ihrkup} is achieved with $(n,k,s)=(7,3,1)$ and $(9,4,1)$. They asked whether there are other values of the parameters with $n\geq 2k+1$ for which Theorem~\ref{thm:ihrkup} is tight. We will show that the answer is yes, by constructing such a family with parameters $(11,5,3)$.

We phrase this problem as an IP as follows. We fix some $n,k,s$. For each $A\in \binom{[n]}{k}$ we introduce a 0-1 variable $x_A$. We force $\F$ to be intersecting as before, by adding for each disjoint pair of sets $A,B$ a constraint $x_A+x_B\leq 1$. To ensure that $\F$ is $s$-subset-regular for each $S\subset [n]$ we add a constraint
$$\sum_{[s]\subset A\in\binom{[n]}{k}} x_A - \sum_{S\subset B\in\binom{[n]}{k}}x_B = 0.$$ Solving this IP directly gives the following construction for $(n,k,s)=(11,5,3)$ in about 30 seconds:

\begin{equation*}
\begin{split}
    \F = &\{\{1,2,3,4,11\}, \{1,2,3,5,6\}, \{1,2,3,7,8\}, \{1,2,3,9,10\}, \{1,2,4,5,10\}, \\ &\{1,2,4,6,7\}, \{1,2,4,8,9\}, \{1,2,5,7,9\}, \{1,2,5,8,11\}, \{1,2,6,8,10\}, \\ &\{1,2,6,9,11\}, \{1,2,7,10,11\}, \{1,3,4,5,7\}, \{1,3,4,6,9\}, \{1,3,4,8,10\}, \\
\end{split}
\end{equation*}
\begin{equation*}
\begin{split}
   ~ ~ ~ ~ ~ ~ ~ ~ ~ &\{1,3,5,8,9\}, \{1,3,5,10,11\}, \{1,3,6,7,10\}, \{1,3,6,8,11\}, \{1,3,7,9,11\}, \\ &\{1,4,5,6,8\}, \{1,4,5,9,11\}, \{1,4,6,10,11\}, \{1,4,7,8,11\}, \{1,4,7,9,10\}, \\ 
\end{split}
\end{equation*}
\begin{equation*}
\begin{split}
   ~ ~ ~ ~ ~ ~ ~ ~  &\{1,5,6,7,11\}, \{1,5,6,9,10\}, \{1,5,7,8,10\}, \{1,6,7,8,9\}, \{1,8,9,10,11\}, \\ &\{2,3,4,5,8\}, \{2,3,4,6,10\}, \{2,3,4,7,9\}, \{2,3,5,7,10\}, \{2,3,5,9,11\}, \\ &\{2,3,6,7,11\}, \{2,3,6,8,9\}, \{2,3,8,10,11\}, \{2,4,5,6,9\}, \{2,4,5,7,11\},
\end{split}
\end{equation*}
\begin{equation*}
\begin{split}
   ~ ~ ~ ~ ~ ~ ~ ~ ~ ~ &   \{2,4,6,8,11\}, \{2,4,7,8,10\}, \{2,4,9,10,11\}, \{2,5,6,7,8\}, \{2,5,6,10,11\}, \\ &\{2,5,8,9,10\}, \{2,6,7,9,10\}, \{2,7,8,9,11\}, \{3,4,5,6,11\}, \{3,4,5,9,10\}, \\ &\{3,4,6,7,8\}, \{3,4,7,10,11\}, \{3,4,8,9,11\}, \{3,5,6,7,9\}, \{3,5,6,8,10\},  
\end{split}
\end{equation*}
\begin{equation*}
\begin{split}
   ~ ~ ~ ~ ~ ~ ~ ~ ~ ~ ~ &\{3,5,7,8,11\}, \{3,6,9,10,11\}, \{3,7,8,9,10\}, \{4,5,6,7,10\}, \{4,5,7,8,9\}, \\ &\{4,5,8,10,11\}, \{4,6,7,9,11\}, \{4,6,8,9,10\}, \{5,6,8,9,11\}, \{5,7,9,10,11\}, \\ &\{6,7,8,10,11\}\}
\end{split}
\end{equation*}

~

\subsection{The Kleitman matching problem}\label{subsec:kleitman}

Let $s\geq 3$ be an integer, and let $k(n,s)$ denote the maximum size of a family $\F\subset 2^{[n]}$ without $s$ pairwise disjoint members. Kleitman~\cite{kleitmanmat} determined $k(n,s)$ for $n\equiv 0$ or $-1$ (mod $s$), see Theorem~\ref{thm:kleitmanmat}. In the case $n\equiv -2$ (mod $s$), the value of $k(n,s)$ was determined by Quinn~\cite{quinn} if $s=3$ and by Frankl and Kupavskii~\cite{franklkupk1,franklkupk2} for all $s$. 

Recall that $k(n+l,s)\geq 2^l k(n,s)$. Indeed, if $\F\subset 2^{n}$ has no $s$ pairwise disjoint members, then neither does $\F' = \{F\subset [n+l]: F\cap [n]\in \F\}$. Kleitman showed~\cite{kleitmanmat} that $k(n,s)=2k(n-1,s)$ if $s$ divides $n$. Motivated by this, Frankl and Tokushige~\cite{frankltoku} made the following conjecture:
\begin{conjecture}[\label{conj:kleitmanmat}\cite{frankltoku}, p.~213]
Let $s\geq 4$. If $n\equiv 1 $ (mod $s$), then $$k(n,s)=4k(n-2,s).$$
\end{conjecture}

\begin{theorem}[\label{thm:kleitmanmat}Kleitman~\cite{kleitmanmat}]
Let $s\geq 2$ be an integer and $\F\subset 2^{[n]}$ a family without $s$ pairwise disjoint members. Then for $n=s(m+1)-\ell$ with $\ell\in [s]$ we have
$$|\F|\leq \frac{\ell - 1}{s}\binom{n}{m} + \sum_{t\geq m+1}\binom{n}{t},$$
and this is sharp for $\ell\in\{1,s\}$.
\end{theorem}

This gives $k(7,4)=120$ and hence in order to disprove Conjecture~\ref{conj:kleitmanmat} our goal is to show $k(9,4)\geq 481$. One can formulate this problem as an IP as follows. As before, we introduce a 0-1 valued indicator variable for every $A\subset [n]$. For each quadruple of pairwise disjoint sets $A,B,C,D$ we add the constraint $x_A+x_B+x_C+x_D\leq 3$. Our goal is then simply to maximize the sum of the variables. 

To speed up the solution of this IP it helps if one makes the heuristic, though certainly unjustified, assumption that $x_A=1$ whenever $|A|\geq 4$ and $x_A=0$ whenever $|A|\leq 1$. Indeed, intuitively it makes sense to include 'large' sets, and so far our family does not even contain three disjoint sets. This restricts the search space to the considerably smaller world $\binom{[n]}{2}\cup \binom{[n]}{3}$.

Note that $480 = 2^9-32 = \binom{9}{\geq 4} + 98$. The LP solver finds a family $\G$ of size $99 $  in $\binom{[9]}{2}\cup \binom{[9]}{3}$ without four pairwise disjoint sets. This gives a counterexample to Conjecture~\ref{conj:kleitmanmat}, as then $\G \cup \binom{[9]}{\geq 4}$ does not contain four pairwise disjoint sets either, and $\left|\G \cup \binom{[9]}{\geq 4}\right|= 4k(7,2) + 1$. The search takes around 2 seconds:

\begin{equation*}
\begin{split}
    \mathcal{G} = \binom{[9]}{3}\cup\left\{A\in\binom{[9]}{2}: |A\cap [2]|\geq 1\right\}.
\end{split}
\end{equation*}

\subsection{Rainbow matchings}\label{subsec:aha}

Aharoni and Howard~\cite{aharhow} considered problems related to rainbow matchings in hypergraphs. Given a collection $\F=(\F_1,\F_2\ldots,\F_k)$ of hypergraphs, a choice of disjoint edges, one from each $\F_i$, is called a \emph{rainbow matching} for $\F$. They made the following conjecture:
\begin{conjecture}[\cite{aharhow}\label{conj:aharhow}]
Let $d>1$, and let $F_1,\ldots,F_k$ be bipartite graphs on the same ground set, satisfying $\Delta(F_i)\leq d$ and $|F_i| > (k-1) d $. Then the system $F_1,\ldots, F_k$ has a rainbow matching. 
\end{conjecture}

To disprove Conjecture~\ref{conj:aharhow} we need to find a collection of bipartite graphs satisfying the bounds above, without a rainbow matching. We phrase this problem as an IP as follows. First, we fix parameters $n,k,d$. Next we fix the partite sets $(L_i,R_i)$ for each of the $k$ bipartite graphs\footnote{We are fairly certain that the phrase ``same ground set'' in Conjecture~\ref{conj:aharhow} only means same vertex set, given the original context in~\cite{aharhow}.}, so that $L_i\dot\cup R_i=[n]$ for each $1\leq i \leq k$. We introduce indicator variables $x_{ab}^{(i)}$ for each $1\leq i \leq k$ and pair of vertices $a\in L_i, b\in R_i$. The maximum degree condition is then a collection of $nk$ simple linear constraints, one for each vertex and each $1\leq i \leq k$. To ensure the system does not have a rainbow matching, for all $k$-tuple of disjoint edges $(a_i,b_i)\in (L_i,R_i)$ for $1\leq i \leq k$, we add a constraint $\sum_{i=1}^k x_{a_i,b_i}^{(i)}\leq k-1$. For the sizes of the graphs, we add linear constraints for all $2\leq i \leq k$ saying that $\sum_{a\in L_i,b\in R_i}x_{ab}^{(i)}\geq (k-1)d+1$. Our goal is then to maximize $\sum_{a\in L_1,b\in R_1}x_{ab}^{(1)}$ and hope that the value of this maximum is greater than $(k-1)d$.

Solving this IP with $n=6$, $k=3$, $d=2$, and partite set $L_1=\{1,2,3\}, L_2=\{2,3,4\}$ and $L_3=\{3,4,5\}$ gives the following counterexample to Conjecture~\ref{conj:aharhow}:
$$F_1 = \{15,16,24,26,34,35\}, \quad F_2=\{14,25,26,35,36\}, \quad F_3=\{13,23,25,46,56\}.$$
Here we have $|F_2|=|F_3|=(k-1)d+1$ and $|F_1|=(k-1)d+2$.

~

\subsection{\label{subsec:jessica}A Tur\'an-type problem in multipartite graphs}

For graphs $G$ and $H$ denote by $\mathrm{ex}(G,H)$ the maximum number of edges in a subgraph of $G$ that contains no copy of $H$. For integers $k,r$ let $kK_r$ denote $k$ vertex-disjoint copies of $K_r$. De Silva et al considered~\cite{jessica} the problem of determining  $\mathrm{ex}(G,H)$ where $H=kK_r$ and $G$ is a complete multi-partite graph. They completely solved this problem when the number of partite sets in $G$ is equal to $r$:
\begin{theorem}[De Silva et al~\cite{jessica}\label{thm:jessica}]
For any integers $k\leq n_1\leq n_2\leq \ldots \leq n_r$,
$$\mathrm{ex}(K_{n_1,\ldots,n_r},kK_r)=\left(\sum_{1\leq i<j\leq r}n_in_j\right)-n_1n_2+n_2(k-1).$$
\end{theorem}
De Silva et al~\cite{jessica} observed that the graph 
$$((n_1+n_2-k+1)K_1\cup K_{k-1,n_3})+K_4$$
does not contain $kK_3$, hence
$$\mathrm{ex}(K_{n_1,n_2,n_3,n_4},kK_3)\geq (n_1+n_2+n_3)n_4+(k-1)n_3.$$
They stated that it is not clear that this is an extremal construction. Using our methods we will show that their intuition was correct, and there exist better constructions.

We phrase the problem as an IP in the standard way. We fix some $n_1,n_2,n_3,n_4,k$ and for each edge $e$ of $K_{n_1,n_2,n_3,n_4}$ we introduce an indicator variable $x_e$. For every collection of $3k$ edges $e_1,\ldots,e_{3k}$ forming a $kK_3$ we include the constraint $\sum_{i=1}^{3k}e_i\leq 3k-1$. 

Solving this IP directly we find that already in the $n_1=n_2=n_3=n_4$ case there exist better constructions. Generalizing the constructions given by the IP solver, we have the following result:
\begin{proposition}
For all integers $k\leq n$, we have
$$\mathrm{ex}(K_{n,n,n,n},kK_3)\geq 4n^2 + (k-1)n.$$
\end{proposition}
\begin{proof}
Let the four partite sets of size $n$ be $A,B,C,D$. Remove all $2n^2$ edges between the pairs $A-B$ and between $C-D$. Between $C$ and $D$ add a copy of $K_{k-1,n}$.
\end{proof}

\section{Concluding remarks}
In this paper we presented a general method that can be used to quickly check whether a conjecture has small counterexamples. Nothing about the method itself is new. We hope to have convinced the reader of the usefulness and versatility of this technique in combinatorics with the various examples in Section~\ref{sec:linbas} and the number of counterexamples to open conjectures in Section~\ref{sec:results}. In practice, the main advantage of writing linear programs is the time saved -- small counterexamples are always found eventually, but it is better to find them in a few minutes rather than a few weeks.

\textbf{Acknowledgement:} The author is indebted to Bernard Lidick\'y who introduced him to linear programming.

\bibliographystyle{amsplain}

\bibliography{mybib}

\newpage

\appendix

\section{Counterexample to Conjecture~\ref{conj:huang2}}

For $n=10$ we have a counterexample to Conjecture~\ref{conj:huang2} of diversity 197. It is given by taking all supersets of elements of the following family $\F_{10}$: 

\begin{tiny}
\begin{equation*}
    \begin{split}
      \F_{10} =\{&\{1,2,6,7\}, \{1,3,6,10\}, \{1,3,7,9\}, \{1,4,5,8\}, \{1,4,6,8\}, \{1,4,7,9\}, \{1,5,8,10\}, \{1,5,9,10\}, \\ &\{2,3,4,5\}, \{2,3,4,10\}, \{2,3,8,9\}, \{2,4,7,10\}, \{2,5,6,9\}, \{2,6,8,9\}, \{2,7,8,10\}, \{3,4,6,10\}, \\ &\{3,5,6,7\}, \{3,5,7,8\}, \{4,6,8,9\}, \{5,6,9,10\}, \{1,2,3,4,5\}, \{1,2,3,4,9\}, \{1,2,3,4,10\}, \{1,2,3,5,8\}, \\ &\{1,2,3,6,7\}, \{1,2,3,6,8\}, \{1,2,3,6,10\}, \{1,2,3,7,9\}, \{1,2,3,8,9\}, \{1,2,3,8,10\}, \{1,2,4,5,7\}, \\ &\{1,2,4,5,8\}, \{1,2,4,5,9\}, \{1,2,4,5,10\}, \{1,2,4,6,7\}, \{1,2,4,6,8\}, \{1,2,4,7,9\}, \{1,2,4,7,10\}, \\ &\{1,2,5,6,7\}, \{1,2,5,6,9\}, \{1,2,5,8,10\}, \{1,2,5,9,10\}, \{1,2,6,7,8\}, \{1,2,6,7,9\}, \{1,2,6,7,10\}, \\ &\{1,2,6,8,9\}, \{1,2,7,8,10\}, \{1,2,7,9,10\}, \{1,3,4,5,8\}, \{1,3,4,6,8\}, \{1,3,4,6,10\}, \{1,3,4,7,9\}, \\ &\{1,3,4,9,10\}, \{1,3,5,6,7\}, \{1,3,5,6,10\}, \{1,3,5,7,8\}, \{1,3,5,7,9\}, \{1,3,5,8,10\}, \{1,3,5,9,10\}, \\ &\{1,3,6,7,8\}, \{1,3,6,7,9\}, \{1,3,6,7,10\}, \{1,3,6,8,10\}, \{1,3,6,9,10\}, \{1,3,7,8,9\}, \{1,3,7,9,10\}, \\ &\{1,4,5,6,8\}, \{1,4,5,7,8\}, \{1,4,5,7,9\}, \{1,4,5,8,9\}, \{1,4,5,8,10\}, \{1,4,5,9,10\}, \{1,4,6,7,8\}, \\ &\{1,4,6,7,9\}, \{1,4,6,8,9\}, \{1,4,6,8,10\}, \{1,4,7,8,9\}, \{1,4,7,9,10\}, \{1,5,6,8,10\}, \{1,5,6,9,10\}, \\ &\{1,5,7,8,10\}, \{1,5,7,9,10\}, \{1,5,8,9,10\}, \{2,3,4,5,6\}, \{2,3,4,5,7\}, \{2,3,4,5,8\}, \{2,3,4,5,9\}, \\ &\{2,3,4,5,10\}, \{2,3,4,6,10\}, \{2,3,4,7,10\}, \{2,3,4,8,9\}, \{2,3,4,8,10\}, \{2,3,4,9,10\}, \{2,3,5,6,7\}, \\ &\{2,3,5,6,9\}, \{2,3,5,7,8\}, \{2,3,5,8,9\}, \{2,3,6,8,9\}, \{2,3,7,8,9\}, \{2,3,7,8,10\}, \{2,3,8,9,10\}, \\ &\{2,4,5,6,7\}, \{2,4,5,6,9\}, \{2,4,5,7,10\}, \{2,4,6,7,10\}, \{2,4,6,8,9\}, \{2,4,7,8,10\}, \{2,4,7,9,10\}, \\ &\{2,5,6,7,9\}, \{2,5,6,7,10\}, \{2,5,6,8,9\}, \{2,5,6,9,10\}, \{2,5,7,8,10\}, \{2,5,8,9,10\}, \{2,6,7,8,9\}, \\ &\{2,6,7,8,10\}, \{2,6,8,9,10\}, \{2,7,8,9,10\}, \{3,4,5,6,7\}, \{3,4,5,6,10\}, \{3,4,5,7,8\}, \{3,4,5,7,9\}, \\ &\{3,4,6,7,10\}, \{3,4,6,8,9\}, \{3,4,6,8,10\}, \{3,4,6,9,10\}, \{3,4,7,8,9\}, \{3,4,7,9,10\}, \{3,5,6,7,8\}, \\ &\{3,5,6,7,9\}, \{3,5,6,7,10\}, \{3,5,6,9,10\}, \{3,5,7,8,9\}, \{3,5,7,8,10\}, \{3,7,8,9,10\}, \{4,5,6,7,8\}, \\ &\{4,5,6,8,9\}, \{4,5,6,9,10\}, \{4,5,7,8,10\}, \{4,6,7,8,9\}, \{4,6,7,8,10\}, \{4,6,8,9,10\}, \{4,7,8,9,10\}, \\ &\{5,6,7,9,10\}, \{5,6,8,9,10\}, \{5,7,8,9,10\}\}
       \end{split}
\end{equation*}
\end{tiny}

\end{document}